\documentclass[12pt]{amsart}
\usepackage{amsthm,amscd,amssymb,amsrefs,mathrsfs}
\usepackage{a4wide}
\usepackage[utf8]{inputenc}
\newtheorem{theorem}{Theorem}[section]
\newtheorem*{theorem*}{Theorem}
\newtheorem{definition}[theorem]{Definition}
\newtheorem{lemma}[theorem]{Lemma}
\newtheorem{corollary}[theorem]{Corollary}
\newtheorem{proposition}[theorem]{Proposition}
\newtheorem{remark}[theorem]{Remark}
\newtheorem{example}[theorem]{Example}
\newtheorem{question}[theorem]{Question}

\newcommand{\wt}[1]{\widetilde{#1}}
\newcommand{\ov}[1]{\overline{#1}}

\newcommand{\lann}{\mathrm{l.ann}}
\newcommand{\rann}{\mathrm{r.ann}}
\newcommand{\ZZ}{\mathbb{Z}}
\newcommand{\Frac}{\mathrm{Frac}}

\title{Ring theoretical properties of affine cellular algebras}
\author[P. Carvalho]{Paula A.A.B. Carvalho}
\address{Departamento de Matemática, Universidade do Porto, Rua Campo Alegre 697, 4169-007 Porto, Portugal}
\email{pbcarval@fc.up.pt}

\author[S. Koenig]{Steffen Koenig}
\address{Institute of Algebra and Number Theory, University of Stuttgart, Pfaffenwaldring 57, 70569 Stuttgart, Germany}
\email{skoenig@mathematik.uni-stuttgart.de}

\author[C. Lomp]{Christian Lomp}
\address{Departamento de Matemática, Universidade do Porto, Rua Campo Alegre 697, 4169-007 Porto, Portugal}
\email{clomp@fc.up.pt}

\author[A. Shalile]{Armin Shalile}
\address{Institute of Algebra and Number Theory, University of Stuttgart, Pfaffenwaldring 57, 70569 Stuttgart, Germany}
\email{shalile@mathematik.uni-stuttgart.de}

\begin{document}
\maketitle

\begin{abstract}
As a generalisation of Graham and Lehrer's cellular algebras, {\it affine cellular algebras} have been introduced in \cite{KoenigXi} in order to treat affine versions of diagram algebras like affine Hecke algebras of type A and  affine Temperley–Lieb algebras in a unifying fashion. 
Affine cellular algebras include Kleshchev's graded quasihereditary algebras, Khovanov-Lauda-Rouquier algebras and various other classes of algebras. In this paper we will study ring theoretical properties of affine cellular algebras.  We show that any affine cellular algebra $A$ satisfies a polynomial identity.
Furthermore, we show that $A$ can be embedded into its asymptotic algebra if the occurring commutative affine $k$-algebras $B_j$ are reduced and the determinants of the swich matrices are non-zero divisors. As a consequence, we show that the Gelfand-Kirillov dimension  of $A$ is less than or equal to the largest Krull dimension of the algebras $B_j$ and that equality hold, in case all affine cell ideals are idempotent or if the Krull dimension of the algebras $B_j$ is less than or equal to $1$. Special emphasis is given to the question when an affine cell ideals is idempotent, generated by an idempotent or finitely generated.
\end{abstract}
\section{Introduction}

 {\em Affine cellular algebras} have been introduced in \cite{KoenigXi} as a generalisation of Graham and Lehrer's cellular algebras. Affine versions of diagram algebras like affine Hecke algebras of type A and  affine Temperley–Lieb algebras are examples of affine cellular algebras. In this paper we will study ring theoretical properties of affine cellular algebras. Recall that an affine cellular algebra $A$ over a commutative Noetherian ring $k$ has a chain of ideals $0=J_{-1}\subset J_0\subset J_1 \subset \cdots \subset J_n=A$, such that $J_j/J_{j-1}$ is an affine cell ideal of $A/J_{j-1}$ and as such is isomorphic, as an $A/J_{j-1}$-bimodule, to a generalised matrix ring $\wt{M_{m_j}(B_j)}$ over some affine commutative $k$-algebra whose multiplication is deformed by a swich matrix $\psi_j \in M_{m_j}(B_j)$. 
Since the publication of \cite{KoenigXi}, several classes of algebras, like the Khovanov-Lauda-Rouquier algebras, Kleshchev's graded quasihereditary algebras, the affine Birman-Murakami-Wenzl algebras, affine Brauer algebras, affine q-Schur algebras and BLN-algebras were shown to be affine cellular (see \cite{Cui_BMW, Cui_Brauer, Cui_BLN, GuilhotMiemietz,KeshchevLoubert,Kleshchev, KleshchevLoubertMiemietz, Nakajima}). These classes of algebras are by definition subclasses of affine cellular algebras and contain other interesting examples, like Kato's geometric extension algebras (see \cite{Kleshchev}*{10.2}). Although it was shown that many algebras are affine cellular, their ring theoretical structure has not been studied in much detail apart from \cite{KoenigXi}.

Affine cellular algebras are built up by affine cell ideals, which will be studied first.  For any element $\psi \in M_n(B)$ of the $n\times n$-matrix ring $M_n(B)$ over a commutative affine $k$-algebra $B$, with $k$ a commutative Noetherian ring, the generalised matrix ring is the associative (possibly non-unital) ring $\wt{M_n(B)}$ which is, as a $k$-module, equal to $M_n(B)$ but whose multiplication is deformed by setting $a*b:=a\psi b$, for elements $a,b \in \wt{M_n(B)}$. 
An affine cell ideal $J$  of an algebra $A$ is isomorphic as a ring to a generalised matrix ring $\wt{M_n(B)}$.
By \cite{KoenigXi}*{Theorem 4.1}, idempotent affine cell ideals are  important for the understanding of the representation theory of affine cellular algebras.
We show in Theorem \ref{endomorphismring_cellideal}  and Proposition \ref{Proposition_finitelygenerated} that
\begin{enumerate}
\item[(1)] $J$ is an idempotent ideal if and only if $B$ is generated (as an ideal) by the entries of $\psi$. In this case
 $J$ is a finitely generated left and right ideal of $A$.
\item[(2)] $J$ is generated by an idempotent in $A$ if and only if $J$ is a principal left ideal of $A$ if and only if $\mathrm{det}(\psi)$ is invertible in $B$. In this case, $A$ decomposes as the ring direct product $A\simeq A/J \times M_n(B)$.
\item[(3)] $\mathrm{End}( _AJ) \simeq M_n(B)$ if $J$ is idempotent or if $J$ contains an element that is a central non-zero divisor in $J$. The latter case is fulfilled in case $\mathrm{det}(\psi)$ is a non-zero divisor in $B$.
\end{enumerate}
Statement (1) gives an alternative description of idempotent affine cell ideals compared with the equivalent conditions found in \cite{KoenigXi}*{Theorem 4.1}. 
Statement (2) clarifies the relation between an affine cell ideal being generated by an idempotent and being an idempotent ideal. Statement (3) has been shown in \cite{KoenigXi}*{Theorem 4.3} for idempotent affine cell ideals with $B$ having zero Jacobson radical. In Statement (3), i.e. Theorem \ref{endomorphismring_cellideal}, the assumption on $B$ is removed. Moreover, an alternative condition (to the idempotence of $J$) is offered to guarantee that 
$\mathrm{End}( _AJ) \simeq M_n(B)$ (as well as $\mathrm{End}( _A\Delta) \simeq B$, see Section 3 for the definition of $\Delta$). 
Example \ref{idempotent_zero_determinant} shows that there are idempotent affine cell ideals with $\mathrm{det}(\psi)=0$, while Example \ref{example_affineTL} shows that some Temperley-Lieb algebras contain non-idempotent affine cell ideals $J$ such that $\mathrm{det}(\psi)$ is a non-zero divisor in $B$.

Our main result is Theorem \ref{properties_aca} which not only shows that affine cellular algebras satisfy a polynomial identity, but also offers a condition to embed an affine cellular algebra $A$ in its asymptotic algebra, which by definition is the direct product of matrix rings over the commutative affine $k$-algebras $B_j$ occurring in the cell structure of $A$:

\begin{theorem*}
	Let $A$ be an affine cellular algebra with cellular structure
	$$ 0=J_{-1} \subset J_0 \subset J_1 \subset \cdots \subset J_n = A,$$
	and $J_{j}/J_{j-1} = \wt{M_{n_j}(B_j)}$ for affine commutative $k$-algebras $B_j$, matrices $\psi_j\in M_{n_j}(B_j)$ and centre $c(A)$.
	Let $m=\mathrm{min}\{ k \mid \rann_{A/J_{k-1}}(J_k/J_{k-1})=0\}$.
	Then
	\begin{enumerate}
		\item $A$ satisfies a polynomial identity.
		\item The following statements are equivalent:
			\begin{enumerate}
				\item[(a)] $A/J_j$ is semiprime for all $j=0,1,\ldots, m-1$;
				\item[(b)] $B_j$ is reduced and $\mathrm{det}(\psi_j)$ is not a zero divisor in $B_j$ for all $j=0,1,\ldots, m-1$;
				\item[(c)] $\Phi: A \rightarrow \mathrm{End}\left(_{A/J_{m-1}}{J_m/J_{m-1}}\right) \times \cdots \times  \mathrm{End}(_{A}{J_0})$ is an embedding  and $B_j$ is reduced for all $j=0,1,\ldots, m-1$.
			\end{enumerate}
		In either of these cases 
		\begin{enumerate}
			\item[(i)] $\mathrm{End}(_{A/J_{j-1}}{J_j/J_{j-1}})\simeq M_{n_j}(B_j)$ for all $j=0,1,\ldots, m-1$;
			\item[(ii)] $c(A) = \Phi^{-1}\left( B_m \times \cdots \times B_0\right).$
		\end{enumerate} 
	\item If $\mathrm{det}(\psi_j)$ is invertible in $B_j$ for all $j$, then $A$ is isomorphic to its asymptotic algebra.
\end{enumerate}
\end{theorem*}
The embedding $\Phi:A\rightarrow M_{n_m}(B_m) \times \cdots \times M_{n_0}(B_0)$ allows to estimate the Gelfand-Kirillov dimension of $A$ (see Corollary \ref{GKdim}):
$$GKdim(A) \leq \mathrm{max} \{ Kdim(B_0), \ldots, Kdim(B_m)\},$$ 
where $Kdim(B)$ denotes the Krull dimension of a commutative affine $k$-algebra $B$. Equality holds if $J_{j}/J_{j-1}$ is a finitely generated left ideal of $A/J_{j-1}$ for all $j\leq m$. The latter condition is fulfilled if the affine cell ideals $J_{j}/J_{j-1}$ are idempotent or if $Kdim(B_j)\leq 1$.

Moreover, under the equivalent conditions of Theorem \ref{properties_aca}(2) the centre $c(A)$ of an affine cellular algebra $A$ is Noetherian if and only if $A$ is left (and right) Noetherian and finitely generated over $c(A)$ (see Corollary \ref{NoetherianCentre}). Standard facts on PI-algebras (see \cite{McConnellRobson}*{13.10.3, 13.10.7}) allow to conclude that simple left $A$-modules over an affine cellular affine $k$-algebra $A$ are finite dimensional over $k$ (compare with \cite{KoenigXi}*{Theorem 3.12}). Moreover, any affine cellular affine $k$-algebra is a Jacobson ring, i.e. prime ideals are intersections of (one-sided) maximal ideals, and has finite classical Krull dimension, i.e. there exists an upper bound for the lengths of chains of prime ideals.

The paper is organised as follows: Elementary facts about swich algebras, by which we mean rings $\wt{R}$ whose underlying abelian group stem from a unital associative ring $R$ and whose multiplication is deformed by a swich element $\psi \in R$, i.e. $a*b = a\psi b$, for $a,b\in R$, are proved in Section 2. Proposition \ref{GeneralisedMatrixRings} specialises to general matrix rings $\wt{M_n(B)}$ and is the main result of this section. Section 3 introduces affine cell ideals as defined in \cite{KoenigXi} and applies the results of Section 2. The main difference between Section 2 and Section 3 is that an affine cell ideal $J$ of an algebra $A$, apart from being a generalised matrix ring, is also an ideal whose $A$-bimodule structure matters. For instance it is possible that $J$ is not finitely generated, as a module over itself, although it is finitely generated as a left $A$-module. Theorem \ref{endomorphismring_cellideal} and Proposition \ref{Proposition_finitelygenerated} are the main results, while Proposition \ref{propositon_construction} offers a method to realise a generalised matrix ring as an affine cell ideal of an algebra $A$, whose $A$-bimodule structure is controlled by a group action on the rows and columns of the matrices. The main result of the paper,  Theorem \ref{properties_aca}, is proved in Section 4 and its consequences regarding the Gelfand-Krillov dimension of an affine cellular algebra is mentioned. The paper finishes with a section on the Noetherianess of affine cellular algebras and several open questions.

All rings in this paper are considered to be associative, but not necessarily unital. For any ring $R$ and subset $X$ of $R$, we denote the left annihilator of $X$ in $R$ by $\lann_R(X)=\{a\in R \mid ax=0,\:  \forall x\in X\}$  and the right annihilator of $X$ in $R$ by $\rann_R(X)=\{a\in R \mid xa=0,\:  \forall x\in X\}$. The centre of a ring $R$ is denoted by $c(R)$, while $M_n(R)$ denotes the ring of $n\times n$-matrices over $R$. The matrices $E_{ij}\in M_n(R)$, whose $(i,j)$-th entry is $1$ and $0$ elsewhere, are called the {\em  matrix units} of $M_n(R)$. Our main ring theoretic reference is the book \cite{McConnellRobson}. 

\section{Swich Algebras}

Let $R$ be any associative unital ring.  Any  element $\psi \in R$ allows to ``deform'' the multiplication on $R$ to yield a new ring structure on the additive group $(R,+)$ by defining a new (associative) multiplication as $a*b = a\psi b$ for any $a,b \in R$. We denote this new ring by $\wt{R}$ if $\psi$ is understood or alternatively as a pair $(R,\psi)$. Then $\wt{R}$ is an associative not necessarily unital ring. Furthermore, there are two ring homomorphisms:
$$\varphi: \wt{R} \rightarrow R, \qquad a \mapsto a\psi, \qquad \mbox{ and } \qquad \varphi': \wt{R} \rightarrow R, \qquad a \mapsto \psi a,$$
which satisfy $\varphi(a)b = a*b = a \varphi'(b)$ for any $a,b \in R$. The kernels of $\varphi$ resp. $\varphi'$ are square-zero ideals  of $\wt{R}$ and coincide with the left resp. right annihilator of $\psi$ in $R$, i.e. $\mathrm{Ker}(\varphi)= \mathrm{l.ann}_R(\psi)$ and $\mathrm{Ker}(\varphi')= \mathrm{r.ann}_R(\psi)$. It is not difficult to see that $\varphi$ (resp. $\varphi'$) is injective if and only if $\psi$ is not a right (resp. left) zero divisor in $R$ and that $\varphi$ (resp. $\varphi'$) is surjective if and only $\psi$ has a left (resp. right) inverse in $R$.

%%% semiprime
An associative (not necessarily unital) ring $R$ is called {\em semiprime} if it does not contain any non-zero nilpotent ideal. A commutative ring $R$ is semiprime if and only if $R$ is reduced, i.e. $R$ has no non-zero nilpotent element.

\begin{lemma}\label{Lemma_semiprime}
$\wt{R}$ is semiprime if and only if $R$ is semiprime and 
 $\psi$ is neither a left nor a right zero divisor in $R$.
\end{lemma}

\begin{proof} 
The element $\psi$ is neither a left nor a right zero divisor in $R$ if and only if $\mathrm{Ker}(\varphi)=\mathrm{Ker}(\varphi')=\{0\}.$ Suppose  that $\wt{R}$ is semiprime.  Then the nilpotent ideals  $\mathrm{Ker}(\varphi)$ and  $\mathrm{Ker}(\varphi')$ have to be zero. 
If $I$ is a square-zero ideal of $R$, then  $I*I=I\psi I \subseteq I^2=0$. As $I$ is an ideal of $\wt{R}$ and $\wt{R}$ is semiprime, $I=0$, showing that $R$ has to be semiprime.

Suppose that $R$ is semiprime and that the kernels of $\varphi$ and $\varphi'$ are zero. Let $I$ be an ideal of $\wt{R}$ with $I*I=0$ and consider the induced ideal $I'=R\varphi(\varphi'(I))R$ of $R$. Since $$(I')^2 = R\varphi(\varphi'(I))R\varphi(\varphi'(I))R = \left(R * I * R\right)* \left( I * R\right) \subseteq I*I = 0$$ and $R$ semiprime, we get $I'= 0$. Thus, as $R$ is unital, $\varphi(\varphi'(I))=0$  and as $\varphi$ and $\varphi'$ are injective, $I=0$, i.e. $\wt{R}$ is semiprime.
\end{proof}

%%% endomorphism ring
Now let us consider the endomorphism ring $\mathrm{End}( _{\wt{R}}\wt{R})$ of $\wt{R}$ as left $\wt{R}$-module. The map  $\rho:R\rightarrow \mathrm{End}( _{\wt{R}}\wt{R})$ given by right multiplication of $R$ on $\wt{R}$, i.e. by the map $a\mapsto \rho_a:[b\mapsto ba]$ for all $a,b \in R$, is an injective ring homomorphism because for any $a,b,c \in R$ one has  $c*\rho_a(b) = c\psi b a = (c*b)a = \rho_a(c*b)$ showing that $\rho_a \in \mathrm{End}(_{\wt{R}}\wt{R})$. Moreover, $\rho$ is injective as $R$ is unital and hence $\rho_a =0$ implies  $a=1a=\rho_a(1)=0$.

\begin{proposition}\label{endomorphismring1}
$\rho: R\rightarrow  \mathrm{End}(_{\wt{R}}\wt{R})$ is an isomorphism whenever  $\wt{R}$ is idempotent or if $\wt{R}$ contains a central non-zero divisor.
\end{proposition}

\begin{proof}
Suppose $\wt{R}$ is idempotent, i.e. $R*R = R$. Then there exist $x_i, y_i \in R$ such that  $1=\sum x_i*y_i$. Hence 
$$ f(b) = \sum f(b x_i*y_i) = \sum bx_i * f(y_i) = b f\left(\sum x_i*y_i\right)=b f(1) = \rho_{f(1)}(b)$$
for any  $f\in \mathrm{End}(_{\wt{R}}\wt{R})$ and $b\in R$. Thus, $\rho:R\rightarrow \mathrm{End}(_{\wt{R}}\wt{R})$ is bijective.

Let  $c$ be a central element of $\wt{R}$. Then $c\psi=c*1=1*c=\psi c$ and for all $b\in R$:
$$ c*f(b)=f(c*b)=f(b\psi c) = f(bc\psi)=bc*f(1) = bc\psi f(1)= b*cf(1)=c*bf(1).$$
 Thus, $c*\left( f - \rho_{f(1)}\right)(b)=0$ implies $f=\rho_{f(1)}$ provided $c$ is a non-zero divisor in $\wt{R}$. Also in this case  $\rho$ is bijective.

\end{proof}

\begin{remark}
It is possible that $\rho:R\rightarrow \mathrm{End}( _{\wt{R}}\wt{R})$ is not surjective as the following example shows.
Let $S$ be a unital ring, $R=S\times S$ and $\wt{R}=(R,\psi)$ with $\psi=(1,0)$. For any  $ g\in \mathrm{End}(_{\ZZ}S)$ we define $\tilde{g} \in  \mathrm{End}(_{\wt{R}}\wt{R})$  by  $\tilde{g}(x,y)=(x,g(y))$ for all $(x,y)\in \wt{R}$. Then $\tilde{g}$ is left $\wt{R}$-linear since
$$\tilde{g}( (x,y)*(x',y'))= \tilde{g}(xx',0) = (xx',0)=(x,y)*(x',g(y'))=(x,y)*\tilde{g}(x',y')$$ for all $(x,y),(x',y')\in R$. 
The map $g\mapsto \tilde{g}$ yields an injective ring homomorphism from $ \mathrm{End}(_{\ZZ}S)$  to  $\mathrm{End}(_{\wt{R}}\wt{R})$. Suppose that there exists  $g\in \mathrm{End}(_{\ZZ}S)$ that is not left $S$-linear, then  $\tilde{g}\not\in\mathrm{Im}(\rho)$ because if $\tilde{g}=\rho_{(a,b)}$, then $(0,g(x))=\tilde{g}(0,x)=\rho_{(a,b)}(0,x)=(0,xb).$ Hence $g(x)=xb$ for all $x\in S$ which shows that $g$ would be left $S$-linear. Therefore, we conclude that $\rho:R\rightarrow \mathrm{End}(_{\wt{R}}\wt{R})$ is not surjective whenever $\mathrm{End}(_{\ZZ}S)\neq \mathrm{End}(_{S}S)$.  Moreover, if $S$  is commutative but  $\mathrm{End}(_{\ZZ}S)$ is not, then $R$ cannot be isomorphic to $\mathrm{End}(_{\wt{R}}\wt{R})$ as $R$ is commutative and $\mathrm{End}(_{\wt{R}}\wt{R})$ contains a subring isomorphic to the non-commutative ring $\mathrm{End}(_{\ZZ}S)$.
\end{remark}

%%% PI Rings %%%
A {\em polynomial identity} on a ring $R$ is a polynomial $f$ with integer coefficients in non-commuting variables $x_i$ such that  $f(a_1, \ldots, a_n)=0$ for any substitutions $x_i=a_i\in R$. More precisely, given a unital ring $R$ and an element $a\in R^n$ for some $n\geq 1$, there exist, by the universal property of free algebras, a unique  unital ring homomorphism $\epsilon_a^R: \ZZ\langle x_1, \ldots, x_n\rangle \rightarrow R$ with $\epsilon_a^R(x_i)=a_i$.  It is common to write $f(a_1,\ldots, a_n) := \epsilon_a^R(f)$ for an element $f\in \ZZ\langle x_1, \ldots, x_n\rangle$. An element $f \in \ZZ\langle x_1, \ldots, x_n\rangle$ is an identity for $R$ if $\epsilon_a^R(f)=0$ for all $a\in R^n$, i.e. if  $f(a_1,\ldots, a_n)=0$ for all substitutions $x_i=a_i\in R$ (see  \cite{McConnellRobson}*{13.1.2}). 
In case of a non-unital ring $R$ and $a\in R^n$ for some $n\geq 1$, there exists a non-unital ring homomorphism $\epsilon_a^R: \ZZ^+\langle x_1, \ldots, x_n\rangle \rightarrow R$ with $\epsilon_a^R(x_i)=a_i$, where $\ZZ^+\langle x_1, \ldots, x_n\rangle$ denotes the non-unital free associative $\ZZ$-algebra, i.e. the ideal of $\ZZ\langle x_1, \ldots, x_n\rangle$ generated by the indeterminates $x_i$. An element $f \in \ZZ^+\langle x_1, \ldots, x_n\rangle$  is an identity for $R$ if $\epsilon_a^R(f)=0$ for all $a\in R^n$, i.e. if all substitutions $f(a_1,\ldots, a_n)$ are zero, for $a_1, \ldots, a_n\in R$ (see  \cite{DrenskyFormanek}*{1.2.3}). 

Such an element $f$, either in $\ZZ\langle x_1, \ldots, x_n\rangle$  or $\ZZ^+\langle x_1, \ldots, x_n\rangle$ depending on $R$ being unital or not, is called a {\em polynomial identity} for $R$. A polynomial identity is called {\em monic} if at least one of the words of highest degree in the support of $f$ has coefficient $1$. A ring $R$ is called a {\em polynomial identity ring} ({\em PI-ring}) if $R$ satisfies some monic polynomial. 

From   \cite{McConnellRobson}*{13.1.7(iv)} it is known, that if $N$ is a nilpotent ideal of a ring $R$ such that $R/N$ is a PI-ring, then $R$ is a PI-ring. Considering the map $\varphi: \wt{R}\rightarrow R$ for some unital ring $R$, $\psi \in R$ and  $\wt{R}=(R,\psi)$, we can therefore conclude that $\wt{R}$ is a (not necessarily unital) PI-ring, if $R$ is PI. 

\begin{lemma}\label{lemma_PIRings}
If $R$ satisfies the monic polynomial identity $f$, then $\wt{R}$ satisfies the monic polynomial identity $f^2$.
 \end{lemma}

\begin{proof}
Define $\varphi_n: \wt{R}^n \rightarrow R^n$ as $\varphi_n(a)=(\varphi(a_1),\ldots, \varphi(a_n))$ for all  $a=(a_1, \ldots, a_n) \in  \wt{R}^n$.  Then $\varphi\circ \epsilon_a^{\wt{R}} = \epsilon_{\varphi_n(a)}^R$ holds because $\varphi ( \epsilon_a^{\wt{R}}(x_i)) = \varphi(a_i)= \epsilon_{\varphi_n(a)}^R(x_i)$ for all $i$ (and the fact that $\varphi, \epsilon_a^{\wt{R}}$ and $\epsilon_{\varphi_n(a)}^R$ are ring homomorphisms).
Suppose $R$ satisfies a monic polynomial $f\in \ZZ\langle x_1, \ldots, x_n\rangle$, then 
$\varphi\left(\epsilon_a^{\wt{R}}(f)\right)=\epsilon_{\varphi_n(a)}^R(f)=0$, i.e. $\epsilon_a^{\wt{R}}(f)\in \mathrm{Ker}(\varphi)$ for all $a\in \wt{R}^n$.
Hence  $\epsilon_a^{\wt{R}}(f^2) = \epsilon_a^{\wt{R}}(f)^2=0$ for all $a\in \wt{R}^n$ as $\mathrm{Ker}(\varphi)^2=0$ and as $\epsilon_a^{\wt{R}}$ is a ring homomorphism.  Thus $\wt{R}$ satisfies $f^2$, which is monic.
\end{proof}
 
The standard polynomial identities are defined for all $n>1$ as 
$$s_n = \sum_{\sigma \in S_n} \mathrm{sgn}(\sigma) x_{\sigma(1)} x_{\sigma(2)} \cdots x_{\sigma(n)}.$$ 
A ring is commutative if and only if it satisfies $s_2=x_1x_2-x_2x_1$. The Amitsur-Levitzki Theorem \cite{McConnellRobson}*{13.3.3(ii)} states that the ring of $n\times n$ matrices $M_n(B)$ over a commutative ring $B$ satisfies the standard identity $s_{2n}$. 

Suppose we are given a ring $B$, $n\geq 1$ and $\psi \in M_n(B)$. The ring $\wt{M_n(B)}=(M_n(B),\psi)$ is called a {\em generalised matrix ring} over $B$ (see \cite{Brown1955}).

\begin{proposition}\label{GeneralisedMatrixRings} Let $B$ be a commutative $k$-algebra over a commutative Noetherian ring $k$, $n>0$ and $\psi \in M_n(B)$. Set $J:=\wt{M_n(B)}=(M_n(B),\psi)$ and let $I=\langle \psi_{ij} \mid 1\leq i,j\leq n\rangle$ be the ideal of $B$ generated by the entries of $\psi$.
 \begin{enumerate}
 	\item $J$ satisfies the monic polynomial identity  $s_{2n}^2$.
 	\item  $J$ is a semiprime  ring if and only if $B$ is reduced and $\mathrm{det}(\psi)$ is not a zero divisor in $B$.
 	\item The adjoint matrix $\psi^+$ of $\psi$ is a central element in $J$. 
 	\item If $\mathrm{det}(\psi)$ is not a zero divisor in $B$, then $ \mathrm{End}( _JJ) \simeq M_n(B)$.
 	\item $J*J =\sum_{i,j=1}^n J*E_{ij}=M_n(I)$.
 	\item $J$ is idempotent if and only if $I=B$. In this case 
	\begin{enumerate}
	  \item[(i)] $J$ is a finitely generated left $J$-module;
	  \item[(ii)] $\mathrm{End}( _JJ)\simeq M_n(B)$ and
	  \item[(iii)] $\mathrm{End}( _J\Delta) \simeq B$, where $\Delta = \bigoplus_{j=1}^n BE_{j1}.$
	  \end{enumerate}
 	\item $J$ is generated as a left ideal over itself by a (central) idempotent if and only if $J$ is a cyclic left $J$-module if and only if $\mathrm{det}(\psi)$ is invertible in $B$.
 \end{enumerate}
\end{proposition}

\begin{proof} 
	(1) $M_n(B)$ satisfies $s_{2n}$ by the Amitsur-Levitzki Theorem and, by Lemma \ref{lemma_PIRings}, $J=\wt{M_n(B)}$ satisfies $s_{2n}^2$.
	
	(2) By Lemma \ref{Lemma_semiprime}, $J$ is semiprime if and only if $M_n(B)$ is semiprime and $\psi$ is neither a left nor right zero-divisor in $M_n(B)$. By \cite{Brown}*{Theorem 9.1}, $\psi$ is a left or right zero divisor in $M_n(B)$ if and only if the determinant $\mathrm{det}(\psi)$ is a zero divisor in $B$. Moreover,  $M_n(B)$ is semiprime if and only if $B$ is reduced since the ideals of $M_n(B)$ are precisely the ideals of the form $M_n(I)$ for an ideal $I$ of $B$.
	
	(3) The adjoint matrix $\psi^+$ of $\psi$ satisfies $\psi^+\psi = \mathrm{det}(\psi)\mathbb{I}_n = \psi\psi^+$. Hence $\psi^+ * a = \mathrm{det}(\psi)a = a*\psi^+$ for all $a\in \wt{M_n(B)}$.
	
	(4) If $\mathrm{det}(\psi)$ is not a zero divisor in $B$, then $\psi^+$ is not a zero divisor in $J$ since $a*\psi^+ = \mathrm{det}(\psi) a \neq 0$ for any non-zero $a\in J$. Hence $\psi^+$ is a central non-zero divisor in $J$ and $M_n(B)\simeq \mathrm{End} (_JJ)$ by Proposition \ref{endomorphismring1}. 
	
	(5) Let $\psi_{ij}$ be an entry of $\psi$. Then $\psi_{ij}E_{kl}= E_{ki}*E_{jl}$ for any $i,j,k,l$. Hence $M_n(I)\subseteq \sum_{i,j=1}^n J*E_{ij} \subseteq J*J$. 
	On the other hand $ \left(bE_{ij}\right)*\left(b'E_{kl}\right) = bb'\psi_{jk}E_{il} \in M_n(I)$ for any $b,b'\in B$ and $1\leq i,j,k,l\leq n$, shows $J*J\subseteq M_n(I)$.
	
	(6) By (5), $M_n(I)=J*J\subseteq J=M_n(B)$. Thus $J$ is idempotent if and only if $I=B$. Moreover, $J=J*J=\sum J*E_{ij}$  shows that $J$ is finitely generated as left $J$-module. Proposition  \ref{endomorphismring1}  shows that the right multiplication $\rho: M_n(B) \rightarrow \mathrm{End}( _JJ)$ is an isomorphism of rings.
	
	For any $1\leq j \leq n$, define the left $J$-submodules $\Delta_j = \bigoplus_{i=1}^n BE_{ij}$, which yields the decomposition  $J=\Delta_1\oplus \cdots \oplus \Delta_n$. Set $\Delta=\Delta_1$. 
	For any $b\in B$, the map $\phi_b: \Delta \rightarrow \Delta$ given by $\phi_b(cE_{i1})=cbE_{i1}$, for $cE_{i1}\in \Delta$, is left $J$-linear since $$\phi_b(aE_{ij}*cE_{k1}) = \phi_b(a\psi_{jk}cE_{i1}) = abc\psi_{jk}E_{i1}=aE_{ij}*bcE_{k1}=aE_{ij}\phi_b(cE_{k1})$$
for all $a,b,c\in B$ and $1\leq i,j,k\leq n$. Moreover, $\phi_b\circ \phi_{b'}=\phi_{bb'}$ for any $b, b' \in B$ and $\phi_1=id_\Delta$ shows that $\phi: B \rightarrow \mathrm{End}( _J\Delta)$ is an injective ring homomorphism.
	
	Let $f$ be any $J$-linear endomorphism of $\Delta$. 
	%Then $f$ can be extended to a left $J$-linear map $f'$ of $J$ given by $f'(a_1+\ldots + a_n)=f(a_1)$ for elements $a_k \in \Delta_k$. 
Let $\pi_\Delta:J \rightarrow \Delta$ be the canonical projection and $\epsilon_\Delta: \Delta\rightarrow J$ be the canonical embedding. Then $f'=\epsilon_\Delta\circ f\circ \pi_\Delta$ is an endomorphism of $J$.
	Using the isomorphism $\rho$ from above, there exists a matrix $M=\sum_{i,j} b_{ij} E_{ij} \in M_n(B)$ such that $f'=\rho(M)$. Hence for any  $a=c_1E_{11} + \cdots c_nE_{n1}\in \Delta$,
	$$ f'(a)=f(a)=\sum_{i,j,l=1}^n c_lb_{ij} E_{l1}E_{ij} = \sum_{j,l=1}^n c_lb_{1j} E_{lj}.$$
	Since $f(a)\in \Delta$, the coefficients of $E_{lj}$ must be zero for $j\neq 1$. Since the coefficients $c_l$ were arbitrary, $b_{1j}=0$ for all $j\neq 1$. Hence
	$f(a)=\sum_{l=1}^n c_lb_{11}E_{l1} = \phi_{b_{11}}(a)$, i.e. $\phi: B \rightarrow \mathrm{End}( _J\Delta)$ is surjective and hence an isomorphism.

	(7) If $J$ is generated as a left ideal by an idempotent, then $J$ is a cyclic left $J$-module. Suppose there exists an element  $e\in J$ such that $J=J*e$. Then $1=f\psi e$ for some $f\in J$. Hence $1=\mathrm{det}(\psi)\mathrm{det}(fe)$, shows that $\mathrm{det}(\psi)$ is invertible in $B$. 

	On the other hand assume $d=\mathrm{det}(\psi)$ is invertible in $B$. Set $e=d^{-1}\psi^+$. Then 
	$$a*e = d^{-1}a\psi\psi^+ = a = (d^{-1}\psi^+)\psi a = e*a$$	
	for all $a\in J$. Thus $e$ is a central idempotent that generates $J$ as left and right ideal.
\end{proof}

\begin{corollary}\label{centralelementsB} Let $B$, $J$ and $\psi$ be as above. Suppose that $\mathrm{det}(\psi)$ is not a zero divisor in $B$. 
Then $B\psi^+ = \{b\psi^+  \mid b\in B\} $ is a central subring of $J$ isomorphic to $\wt{B}=(B,\mathrm{det}(\psi))$.
\end{corollary}

\begin{proof}
The map $B\longrightarrow B\psi^+$ sending $b$ to $b\psi^+$ is injective, since for $b\psi^+=0$ one has $b\mathrm{det}(\psi)=b\psi^+\psi=0$ and hence $b=0$. The product of two elements 
$b\psi^+$ and $b'\psi^+$ is given by
$bb'\mathrm{det}(\psi)\psi^+$ showing that $B\psi^+ \simeq (B,\mathrm{det}(\psi))$.
\end{proof}

\begin{example}\label{idempotent_zero_determinant} Let $B$ be any commutative $k$-algebra, $R=M_2(B)$,
 $\psi=E_{11}$  and $J=\wt{R}=(M_2(B),\psi)$. The ideal generated by the entries of $\psi$ is $B$. Hence by Proposition \ref{GeneralisedMatrixRings}(6), the ideal $J$ is idempotent, but not generated by an idempotent element as $\mathrm{det}(\psi)=0$ is not invertible in $B$.  Precisely, if $J$ were a cyclic left $J$-module generated by some element $e\in J$, then $\mathbb{I}_2 = f*e=f \psi e$ for some element $f\in J$. Thus $1=\mathrm{det}(f)\mathrm{det}(\psi)\mathrm{det}(e)$ would contradict $\mathrm{det}(\psi)=0$. Hence $J$ cannot be a cyclic left $J$-module. However, $J=J*E_{11} \oplus J*E_{12}$ as a left $J$-module.
\end{example}

\section{Affine cell ideals}

For the rest of the paper, we will assume that $k$ is a commutative Noetherian ring.
Recall the definition of an affine cell ideal from \cite{KoenigXi}*{Definition 2.1}.

\begin{definition} Let  $A$ be a unitary $k$-algebra with $k$-involution $i$. A two-sided ideal $J$ of $A$ is
called an {\em affine cell ideal} if and only if the following data are given and the following conditions are satisfied:
\begin{enumerate}
	\item The ideal $J$ is fixed by $i$: $i(J)=J$.
	\item There exists a free $k$-module $V$ of finite rank and an affine commutative $k$-algebra $B$ with identity and with a $k$-involution $\sigma$ such that $\Delta:= V \otimes_k B$ is an $A-B$-bimodule, where the right $B$-module structure is induced by that of the right regular $B$-module $B_B$.
	\item There is an $A-A$-bimodule isomorphism $\alpha : J \rightarrow \Delta\otimes_B \Delta'$, where $\Delta'=B\otimes_k V$ is a $B-A$-bimodule with the left $B$-structure induced by $_BB$ and with the right $A$-structure defined  via $i$, that is, $(b\otimes v)a:=\tau(i(a)(v\otimes b))$ for $a\in A$, $b\in B$ and $v\in V$, where $\tau:\Delta \rightarrow \Delta'$ denotes the flip map, such that the following diagram is commutative:
	\[\begin{CD} 
		J @>\alpha>> \Delta\otimes_B \Delta' \\
		@ViVV @VV{v_1\otimes b_1 \otimes_B b_2 \otimes v_2 \rightarrow v_2\otimes \sigma(b_2)\otimes_B \sigma(b_1) \otimes v_1}V\\
		J @>>\alpha> \Delta \otimes_B \Delta'
	\end{CD}\]
\end{enumerate}
\end{definition}
The module $\Delta$ is called the {\em cell lattice} of $J$. 
The module $\Delta\otimes_B \Delta'$ can be identified with $V\otimes B \otimes V$ and property (3) implies that for all $v,v' \in V$ and $b\in B$:
$$ \alpha( i (\alpha^{-1}(v\otimes b \otimes v'))) = v' \otimes \sigma(b) \otimes v.$$

The multiplication on $J$ leads to a multiplication $\bullet$ on $V\otimes B \otimes V$ as follows:
$$ u\bullet w = \alpha( \alpha^{-1}(u)\alpha^{-1}(w) )$$
for all $u,w \in V \otimes B \otimes V$. With this product, $\alpha$ is a ring isomorphism between $J$ and $V\otimes B \otimes V$. Let $\{v_1, \ldots, v_n\}$ be a basis of $V$, $1\leq i,j,s,t\leq n$, $b,c\in B$ and set
$u=v_i\otimes b \otimes v_j$ and $w=v_s\otimes c \otimes v_t$. 
By \cite{KoenigXi}*{Proposition 2.2} there exists a bilinear form $\psi:V\otimes V \rightarrow B$ such that 
$$u \bullet w = (v_i\otimes b \otimes v_j)\bullet (v_s\otimes c \otimes v_t) = v_i \otimes b\ \psi(v_j, v_s) c \otimes
v_t.$$
We identify the bilinear form $\psi$ with the matrix $\psi=(\psi_{ij}) \in M_n(B)$, where  $\psi_{ij}=\psi(v_i,v_j)$, and
consider its generalised matrix ring $\wt{M_n(B)}=(M_n(B),\psi)$. 
Then 
$$ \alpha': V\otimes B\otimes V \longrightarrow \wt{M_n(B)}, \qquad 
v_i\otimes b\otimes v_j \mapsto bE_{ij}$$
is an isomorphism of algebras, where $E_{ij}$ denotes the matrix units.  Set $\ov{\alpha}=\alpha'\circ\alpha:J\simeq \wt{M_n(B)}$. Recall the map $\varphi:\wt{M_n(B)} \rightarrow M_n(B)$ sending a matrix $M$ to $M\psi$ and set $\ov{\varphi}=\varphi\circ \ov{\alpha}$.
 Analogously, one defines $\overline{\varphi'}=\varphi'\circ\ov\alpha:J \rightarrow M_n(B)$ with $a \mapsto \psi\ov\alpha(a)$.
 
\begin{lemma}\label{annihilators}
$\mathrm{Ker}( \ov{\varphi}) = \lann_J(J)$ and $\mathrm{Ker}(\overline{\varphi'})= \rann_J(J).$
In particular, $ \lann_J(J)=0$ if and only if $\rann_J(J)=0$ if and only if $\mathrm{det}(\psi)$ is not a zero divisor in $B$.
\end{lemma}

\begin{proof}
Let $\mathbb{I}_n$ denote the identity matrix in $M_n(B)$ and set $e=\ov{\alpha}^{-1}(\mathbb{I}_n)$.  Then for any $a\in \lann_J(J)$:
$$ \ov{\varphi}(a)=\ov{\alpha}(a)\psi =\ov{\alpha}(a)\psi \mathbb{I}_n =
\ov{\alpha}(a)*\ov{\alpha}(e)=\ov{\alpha}(ae)=0,$$
Thus $\lann_J(J)\subseteq \mathrm{Ker}( \overline{\varphi})$.  Similarly, if $a\in \mathrm{Ker}( \ov{\varphi})$, then for any $b\in J$: 
$$\ov\alpha(ab)=\ov\alpha(a)*\ov\alpha(b)=\ov\alpha(a)\psi\ov\alpha(b)=\ov\varphi(a)\ov\alpha(b)=0.$$
As $\ov{\alpha}$ is an isomorphism, $ab=0$ for all $b\in J$, i.e. $a\in \lann_J(J)$.

Clearly,  $\mathrm{Ker}( \ov{\varphi}) =0$ if and only if $\psi$ is not a right zero divisor in $M_n(B)$. By \cite{Brown}*{Theorem 9.1},  $\psi$ is not a right zero divisor in $M_n(B)$ if and only if $\mathrm{det}(\psi)$ is not a zero divisor in $B$ which proves the last claim.
\end{proof}

The last Lemma shows that  $\mathrm{Ker}(\overline{\varphi})$ is not only an ideal in $J$ but also an ideal in $A$.

Next we will apply Proposition \ref{endomorphismring1} in order to conclude that  $\mathrm{End}(_AJ) \simeq M_n(B)$ provided $J$ is idempotent or contains a non-zero element that is a central non-zero divisor in $J$.

\begin{theorem}\label{endomorphismring_cellideal}
Let $J$ be an affine cell ideal of $A$ with cell lattice $\Delta$. Suppose one of the following conditions hold:
\begin{enumerate}
\item[(a)] $J$ is an idempotent ideal or 
\item[(b)] $J$ contains  an element that is a central non-zero divisor in the ring $J$.
\end{enumerate}
Then $\mathrm{End}(_AJ) = \mathrm{End}(_JJ)\simeq M_n(B)$ and $\mathrm{End}( _A\Delta) \simeq B$.
\end{theorem}

\begin{proof}
Both  conditions on $J$ pass over to $\wt{M_n(B)}$ via $\ov\alpha$.  Hence by Proposition  \ref{endomorphismring1} the right multiplication  $\rho: M_n(B) \rightarrow \mathrm{End}\left(_{\wt{M_n(B)}}{\wt{M_n(B)}}\right)$ is an isomorphism of $k$-algebras.  Furthermore, the isomorphism $\ov{\alpha}: J \simeq \wt{M_n(B)}$ induces an isomorphism $\gamma: \mathrm{End}(_JJ)\simeq \mathrm{End}\left(_{\wt{M_n(B)}}{\wt{M_n(B)}}\right)$ whose inverse map is defined by $\gamma^{-1}(f)=\ov\alpha^{-1}\circ f\circ\ov\alpha$ for all $f\in \mathrm{End}\left(_{\wt{M_n(B)}}{\wt{M_n(B)}}\right)$. Thus 
$\gamma^{-1}\circ \rho: M_n(B) \rightarrow \mathrm{End}(_JJ)$ is an isomorphism of $k$-algebras sending a matrix $m$ to $\ov\alpha^{-1}\circ \rho_m \circ\ov\alpha$. 

We will show that both conditions on $J$ imply $\mathrm{End}(_AJ)=\mathrm{End}(_JJ)$. 

In case $J$ is idempotent for any $x\in J$,  there exist $y_1,\ldots, y_m, z_1,\ldots, z_m \in J$ such that $x=\sum_{i=1}^m y_iz_i$. Hence for any $f\in \mathrm{End}(_JJ)$ and $a\in A$ we get
$$ f(ax) = \sum_{i=1}^m f(ay_i z_i) = \sum_{i=1}^m a y_i f(z_i) = a f\left(\sum_{i=1}^m y_iz_i\right) = af(x).$$

In case $J$ contains a non-zero  element $c$ that is not a left zero divisor in $J$, we have for any $f\in \mathrm{End}(_JJ)$, $x\in J$ and $a\in A$:
$ c ( af(x)-f(ax) ) = ca f(x) - c f(ax) = f(cax-cax ) = 0.$
Hence $af(x)=f(ax)$.

Let $\wt{\Delta} := \bigoplus_{i=1} B E_{i1}  \subseteq \wt{M_n(B)}$
and $\Delta=V\otimes B$, with $V$ a free $k$-module with basis $\{v_1,\ldots, v_n\}$.
Then $\alpha'^{-1}(\wt{\Delta}) = \Delta \otimes v_1$ is a left $A$-submodule of $\Delta\otimes V$ and as such a direct summand. Hence $\ov\alpha^{-1} ( \wt{\Delta} )$ is a direct summand of $J$ as  left $A$-submodule and isomorphic to $\Delta$.
Since $\mathrm{End}(_JJ)=\mathrm{End}(_AJ)$ and since $\ov\alpha^{-1}(\wt{\Delta})$ is a direct summand of $_AJ$, any left $J$-linear endomorphism of $\ov\alpha^{-1}(\wt{\Delta})$ is left $A$-linear. Thus $$\mathrm{End}(_A\Delta)\simeq  \mathrm{End}(_A\ov\alpha^{-1}(\wt{\Delta})) = \mathrm{End}(_J\ov\alpha^{-1}(\wt{\Delta})).$$
The map $\gamma': \mathrm{End}(_J\ov\alpha^{-1}(\wt{\Delta}))\rightarrow \mathrm{End}( _{\wt{M_n(B)}}{\wt{\Delta}})$,
given by $f\mapsto \ov\alpha\circ f \circ \ov\alpha^{-1}$, is an isomorphism of rings. Hence by Proposition \ref{endomorphismring1}, 
$$\mathrm{End}(_A\Delta)\simeq  \mathrm{End}(_J\ov\alpha^{-1}(\wt{\Delta})) \simeq \mathrm{End}\left(_{\wt{M_n(B)}}{\wt{\Delta}}\right) \simeq B.$$
\end{proof}

\begin{remark} 
Theorem \ref{endomorphismring_cellideal}(a) shows that if $J$ is an idempotent ideal, then $B\simeq \mathrm{End}(_A\Delta)$. This isomorphism has been obtained in \cite{KoenigXi}*{Theorem 4.3(2)} under the additional assumption that the radical of $B$ is zero. Proposition \ref{endomorphismring_cellideal} removes this assumption and moreover generalises \cite{KoenigXi}*{Theorem 4.3(2)} by showing that the same conclusion holds if $J$ contains a central non-zero divisor in $J$. Example \ref{example_affineTL}  will show that there exists an affine cell ideal in an affine Temperley-Lieb algebra, that is not idempotent but  contains a central non-zero divisor.
\end{remark}

In the next proposition we will find some sufficient conditions for an affine cell ideal to be finitely generated as left ideal.

\begin{proposition}\label{Proposition_finitelygenerated}
	Let $J$ be an affine cell ideal of $A$ with cell data $B$, $n$ and $\psi$ as above. For any $1\leq i,j\leq n$, let $e_{ij}=\ov{\alpha}^{-1}(E_{ij})\in J$ where $E_{ij}$ are the matrix units of $M_n(B)$. Denote by $I$ the ideal of $B$ generated by the entries of $\psi$ and let $\{ b_\lambda + I \mid \lambda \in \Lambda\}$ be a generating set of $B/I$ as $k$-module.
	\begin{enumerate}
		\item $J^2 = \sum_{i,j=1}^n Je_{ij} \subseteq \sum_{i,j=1}^n Ae_{ij} \subseteq J=J^2 + \sum_{i,j,\lambda} A \ov{\alpha}^{-1}{\left(b_\lambda E_{ij}\right)}.$
		\item If $B/I$ is a finitely generated $k$-module, then $J$ is a finitely generated left $A$-module.
		
		\item $J$ is idempotent if and only if $B=I$.

		\item $J=Ae$ for some (central) idempotent $e$ if and only if $\mathrm{det}(\psi)$ is invertible in $B$. In this case the idempotent can be chosen to be $e=\ov\alpha^{-1}(\mathrm{det}(\psi)^{-1}\psi^+)$ and $A$ decomposes as the ring direct product $A\simeq A/J \times M_n(B)$.

	\end{enumerate}
\end{proposition}

\begin{proof}
	(1)  $J^2 = \sum_{i,j=1}^n Je_{ij}$ follows from 
	Proposition \ref{GeneralisedMatrixRings}(5). 
	Any element of $J$ is of the form $\sum_{i,j=1}^n\ov{\alpha}^{-1}\left( b_{ij} E_{ij}\right)$ for some $b_{ij}\in B$.
	Any $b\in B$ can be written as $ b = \sum_{i,j} c_{ij}\psi_{ij} + \sum_{\lambda} \mu_{\lambda} b_\lambda$ for some $c_{ij}\in B$ and only finitely many non-zero $\mu_{\lambda}\in k$.
	Hence for any $1\leq s,t\leq n$ one has 
\begin{eqnarray*}
\ov{\alpha}^{-1}\left(bE_{st}\right) &=& \sum_{i,j} \ov{\alpha}^{-1}\left(c_{ij}\psi_{ij}E_{st}\right) + \sum_{\lambda} \mu_{\lambda} \ov{\alpha}^{-1}\left(b_\lambda E_{st}\right)\\
&=& \sum_{i,j} \ov{\alpha}^{-1}\left(c_{ij}E_{si}\right) e_{jt}  + \sum_{\lambda} \left(\mu_{\lambda}1_A\right) \ov{\alpha}^{-1}\left(b_\lambda E_{st}\right) \in J^2 + 
\sum_{i,j,\lambda} A \ov{\alpha}^{-1}{\left(b_\lambda E_{ij}\right)}.
\end{eqnarray*} 
	Thus $J=J^2 + \sum_{i,j,\lambda} A \ov{\alpha}^{-1}{\left(b_\lambda E_{ij}\right)}$.

	(2) By (1), $J = \sum_{i,j} Ae_{ij} +  \sum_{i,j,\lambda} A \ov{\alpha}^{-1}{\left(b_\lambda E_{ij}\right)}$. Hence if $B/I$ has a finite generation set as $k$-module, we can choose $\Lambda$ to be finite and $J$ is a finitely generated left $A$-module.
		
	(3) follows from Proposition \ref{GeneralisedMatrixRings}(6). 
	
	(4)  If $J=Ae$ for some idempotent $e\in J$, then for any $x\in J$ there exist $a\in A$ with $x=ae$. Hence $x=xe$ and $J=Je$. Thus $\wt{M_n(B)}=\wt{M_n(B)}*\ov\alpha(e)$. By Proposition \ref{GeneralisedMatrixRings}(7), $d=\mathrm{det}(\psi)$ is invertible. On the other hand Proposition \ref{GeneralisedMatrixRings} also says that $e=\ov\alpha^{-1}(d^{-1}\psi^+)$ is a central idempotent if $d$ is invertible.
		
	Moreover, the map $\varphi:\wt{M_n(B)}\rightarrow M_n(B)$ is surjective since $a=d^{-1}a\psi^+\psi= \varphi(d^{-1}a\psi^+)$. Hence $Ae=J\simeq M_n(B)$. Since $e$ is a central idempotent, $1-e$ is also one and leads to the  decomposition $A=A(1-e) \oplus Ae \simeq A/J\times M_n(B)$.

\end{proof}

\begin{corollary}
Let $J$ be an affine cell ideal of $A$ with cell data $B$, $n$ and $\psi$ as above. If $k$ is a field, $B$ a commutative affine domain over $k$ of Krull dimension less than or equal to one and $\psi$ is non-zero, then $J$ is a finitely generated left $A$-module.
\end{corollary}

\begin{proof}
Since $\psi$ is non-zero, the ideal $I$ generated by the non-zero entries of $\psi$ in $B$ is a non-zero ideal. If $B$ has Krull dimension $0$, then $B$ is a finite field extension of $k$ and $B=I$. If $B$ has Krull dimension $1$, then $B/I$ has Krull dimension zero as $I$ is non-zero. Since $B/I$ is a commutative Artinian affine $k$-algebra, it is finite dimensional over $k$ (see \cite{McConnellRobson}*{13.10.3}). Now the claim follows by Proposition \ref{Proposition_finitelygenerated}(2).
\end{proof}

%%%%%%%
%%Construction
%\label{example1_affinecell}
Given any affine commutative $k$-algebra $B$, $n\geq 1$ and symmetric matrix 
$\psi \in M_n(B)$, we can construct algebras $A$ that contain $J=\widetilde{M_n(B)} = (M_n(B),\psi)$ as an affine cell ideal. More precisely for any group $G$ and group homomorphism $\rho:G\rightarrow S_n$, the group algebra $B_0=k[G]$ over $k$ acts on $J$ by setting 
$$g\cdot (bE_{ij}) = bE_{\rho(g)(i)j} 
\qquad \mbox{and} 
\qquad bE_{ij}\cdot g = bE_{i\rho(g^{-1})(j)}, 
\qquad \forall g\in G, b\in B, i,j\in\{1,\ldots,n\}.$$
With this action $J$ becomes a $B_0$-bimodule such that the bimodule action of $B_0$ on $J$ is associative with the multiplication in $J$, i.e. for all $g\in G, b,c\in B$ and $1\leq i,j,k,l\leq n$:
$$g\cdot (bE_{ij}*cE_{kl}) = g\cdot (bc\psi_{jk} E_{il}) = bc\psi_{jk}E_{\rho(g)(i)l} = (bE_{\rho(g)(i)j}) * (cE_{kl}) =  (g\cdot (bE_{ij})) * (cE_{kl}).$$
Similarly, one checks $(bE_{ij}*cE_{kl})\cdot g = (bE_{ij})*((cE_{kl})\cdot g).$ Hence $A=B_0\oplus J$ becomes an algebra over $k$ with multiplication given by
$(g, a)(h, a') = (gh, g\cdot a' + a\cdot h + a*a'),$ for all $a,a'\in J$ and $g,h\in G$ such that $J$ is an ideal of $A$. 

Since we assume that $\psi$ is symmetric, the transpose in $M_n(B)$ extends to an involution $\sigma$ of $A$ with $\sigma(g)=g^{-1}$ for all $g\in G$. We check that
$$\sigma(g\cdot E_{ij}) = \sigma(E_{\rho(g)(i)j}) = E_{j\rho(g)(i)} = E_{ji}\cdot g^{-1} = \sigma(E_{ij})\cdot\sigma(g)$$ for all $g\in G$. It is clear that $\sigma(J)=J$. Thus $J$ is an affine cell ideal of $A$. 

\begin{proposition}\label{propositon_construction}
	Let $A=k[G]\oplus J$ with $J=\widetilde{M_n(B)}$ as above and let $I=\langle \psi_{ij}\mid i,j\rangle$ be the ideal of $B$ generated by the entries of $\psi$. Then $J$ is a finitely generated left ideal of $A$ if and only if $B/I$ is a finitely generated $k$-module.
\end{proposition}

\begin{proof}
Let $G$ be a group and $\rho:G\rightarrow S_n$ a group homomorphism such that $k[G]$ acts on $J=(M_n(B),\psi)$ as described above. Set $A=k[G]\oplus J$. Suppose $J$ is generated as left $A$-module by elements $x_1, \ldots, x_m \in J$. For each $1\leq t \leq m$ there exist elements $b^t_{ij}\in B$ such that $x_t = \sum_{i,j} b^t_{ij}E_{ij}$. We claim that $B=I+\sum_{i,t} kb^t_{i1}$ as $k$-module. For any $b\in B$, $bE_{11} \in J$. Thus there exist elements $a_1,\ldots, a_m \in A$ such that $bE_{11} = a_1x_1+\cdots + a_mx_m$. Each of the elements $a_t$ can be written as $a_t = \sum_{g\in G} \lambda_g g + c_t$ for $c_t\in J$ and finitely many non-zero elements $\lambda_g \in k$. Thus
$$ bE_{11} = \sum_{g,i,j,t} \lambda_g g\cdot (b^t_{ij}E_{ij}) + \sum_{t=1}^m c_t*x_t
= \sum_{g,i,j,t} \lambda_g b^t_{ij}E_{\rho(g)(i)j} + \sum_{t=1}^m c_t*x_t.$$
Hence 
$ bE_{11}-\sum_{g,i,j,t} \lambda_g b^t_{ij}E_{\rho(g)(i)j} \in J*J=M_n(I)$. 
Comparing the coefficients of $E_{11}$ we get
$ b - \sum_{g,t} \lambda_g b^t_{\rho(g^{-1})(1)1} \in I$.
Thus $B/I$ is finitely generated as $k$-module. 

The converse follows from Proposition \ref{Proposition_finitelygenerated}(2).
\end{proof}

\begin{example}\label{example_affineTL}
One particular instance of the construction above is the case $B=k[x]$, $n=2$,  $\psi=\left(\begin{array}{cc}q & x \\ x & q\end{array}\right)$ with $q\in k$ and $G=\langle \tau\rangle$ is the infinite cyclic group with group homomorphism $\rho:G\rightarrow S_2$ sending $\tau$ to the cycle $(12)$. Then $B_0=k[G]=k[\tau,\tau^{-1}]$ and $A=k[\tau^{\pm 1}]\oplus \widetilde{M_2(k[x])}$ is isomorphic to the affine Temperly-Lieb algebra with two vertices and parameter $q$. Let $I=\langle q, x\rangle$ be the ideal generated by the entries of $\psi$. Since $x\in I$, $B/I$ is isomorphic to a quotient of $k$ and hence a cyclic $k$-module. Thus, by  Proposition \ref{propositon_construction} or Proposition  \ref{Proposition_finitelygenerated}(2), $J$ is a finitely generated left $A$-module. 
By Propositon \ref{GeneralisedMatrixRings}(6), $J$ is an idempotent ideal if and only if $I=B$, which is equivalent to $q$ being invertible in $k$. Hence over a field $k$, $J$ is  idempotent if and only if $q\neq 0$. Note that $\mathrm{det}(\psi)=q^2-x^2$ is a non-invertible, non-zero divisor in $B$. By Proposition \ref{GeneralisedMatrixRings}(7), $J$ is not generated by an idempotent element.

In case $q$ is not invertible in $k$,  $I \neq B$ and Proposition \ref{GeneralisedMatrixRings}(6) shows that $J$ is not idempotent. However as $\mathrm{det}(\psi)=q^2-x^2$ is a non-zero divisor in $B$, the adjoint matrix $\ov\alpha^{-1}(\psi^+)$ is a central non-zero divisor in $J$. By Theorem \ref{endomorphismring_cellideal}, $\mathrm{End}(_AJ)\simeq M_2(k[x])$.
\end{example}

\begin{example}\label{example_non_finite_cell_ideal}
Consider $B=k[x,y]$ instead of $B=k[x]$ in Example \ref{example_affineTL} and suppose $q=0$, i.e. $\psi=\left(\begin{array}{cc}0 & x \\ x & 0\end{array}\right)$. Then $J = \widetilde{M_2(B)}$ is an example of an affine cell ideal that is not finitely generated as a left $A$-module. Note that although $B_0=k[\tau, \tau^{-1}]$ and $B$ are affine commutative domains, and hence Noetherian (as $k$ is Noetherian), the algebra $A$ is not Noetherian.
\end{example}
 	
\begin{theorem}\label{Theorem_cell_ideals} Let $J$ be an affine cell ideal of a unitary $k$-algebra $A$.
 \begin{enumerate}
  \item $A$ satisfies a polynomial identity if and only if $A/J$ satisfies a polynomial identity.
  \item If $A/J$ is semiprime, then $A$ is semiprime if and only if $B$ is reduced and $\mathrm{det}(\psi)$ is not
a zero divisor in $B$.
 \end{enumerate}
\end{theorem}

\begin{proof}
	(1)  It is clear that if $A$ is a PI-ring, then so is $A/J$. Suppose that $A/J$ is a PI-ring. An element $f \in \ZZ\langle x_1, \ldots, x_n\rangle$ of the form $f = \sum_{\sigma \in S_n} a_\sigma x_{\sigma(1)}\cdots x_{\sigma(n)}$ is called a {\em multilinear polynomial}. By \cite{McConnellRobson}*{13.1.19}, any PI-ring satisfies a monic multilinear polynomial. Let $f(x_1, \ldots, x_n)$ be a monic multilinear polynomial which $A/J$ satisfies. By Proposition \ref{GeneralisedMatrixRings}, $J$  is a PI-ring and satisfies a  monic multilinear polynomial  $g(y_1, \ldots, y_m)$. Define the composition of $f$ and $g$ in $nm$ variables $z_{ij}$, $1\leq i \leq m, 1\leq j \leq n$, as 
	$$h(z_{11},\ldots, z_{1n}, \ldots, z_{m1}, \ldots, z_{mn}) := g\left(f(z_{11}, \ldots, z_{1n}), f(z_{21}, \ldots, z_{2n}), \ldots, f(z_{m1}, \ldots, z_{mn})\right).$$ 
	Then, for any family of elements $(r_{ij})_{1\leq i \leq m, 1\leq j \leq n} \in A$ one has that $f_i:=f(r_{i1}, \ldots, r_{in}) \in J$ for any $i$ as $f$ is an identity for $A/J$ and hence $g(f_1, \ldots, f_m)=0$ as $g$ is an identity for $J$. To see that $h$ is a monic polynomial, we can write $f=\sum_{\sigma\in S_n} a_{\sigma} x_{\sigma(1)} \cdots x_{\sigma(n)}$ and $g=\sum_{\tau\in S_m} b_{\tau} y_{\tau(1)} \cdots y_{\tau(m)}$. Then 
	\begin{eqnarray*}
		h(z_{11}, \ldots, z_{mn})
		&=& \sum_{\sigma_1, \ldots, \sigma_m  \in S_n} a_{\sigma_1}\cdots a_{\sigma_m} g(z_{1\sigma_1(1)}\cdots z_{1\sigma_1(n)}, \ldots, z_{m\sigma_m(1)}\cdots z_{m\sigma_m(n)})\\
		&=&\sum_{\tau \in S_m} \sum_{\sigma_1, \ldots, \sigma_m \in S_n} b_\tau  a_{\sigma_1}\cdots a_{\sigma_m} z_{\tau(1)\sigma_1(1)}\cdots z_{\tau(1)\sigma_1(n)} \cdots z_{\tau(m)\sigma_m(1)}\cdots z_{\tau(m)\sigma_m(n)}
	\end{eqnarray*}
	Since $f$ and $g$ are monic, there are $\sigma' \in S_n$ and $\tau' \in S_m$ with $a_{\sigma'}=1=b_{\tau'}$. Thus, for $\sigma_i:=\sigma'$, also $b_{\tau'} a_{\sigma_1}\cdots a_{\sigma_m}=1$. Moreover, all occurring monomials in this representation of $h$  are different.
	
	(2) If $I$ is a nilpotent ideal of $A$, then $(I+J)/J$ is nilpotent in $A/J$ and therefore must be zero since $A/J$ is semiprime. Hence $I$ is a nilpotent ideal contained in  $J$. If $B$ is reduced and $\mathrm{det}(\psi)$ is not a zero divisor in $B$, then by Propositon \ref{GeneralisedMatrixRings} $\wt{M_n(B)} \simeq J$ is a semiprime ring. Hence $I=0$. 
	
	For the converse suppose that $A$ is semiprime. The square-zero  ideals  $\mathrm{Ker}(\overline{\varphi})=\lann_J(J)$ and $\mathrm{Ker}(\overline{\varphi'}) $ must be zero and $\psi$ is neither a right nor a left zero divisor. By \cite{Brown}*{Theorem 9.1}, $\mathrm{det}(\psi)$ is not a zero divisor in $B$. 
		Let $b\in B$ be an element such that $b^2=0$ and consider the elements  $b\mathbb{I}_n \in M_n(B)$ and $a=\ov\alpha^{-1}(b\mathbb{I}_n)\in J$. The set $aJ$ is a right ideal of $A$ and satisfies 	
	$$(aJ)^2 = \ov\alpha^{-1}\left( b\mathbb{I}_n *  \wt{M_n(B)} * b\mathbb{I}_n *  \wt{M_n(B)} \right) \subseteq \ov\alpha^{-1}\left( b^2  \wt{M_n(B)}\right) = 0.$$ 
	As $A$ is semiprime, $aJ=0$, i.e. $a\in \lann_J(J) = \{0\}$. Hence $b=0$, showing that $B$ is reduced.
\end{proof}

Given an ideal $J$ of a ring $R$ we will look at embeddings of $R$ into $R/J \times \mathrm{End}(_RJ)$, which will be later used to show how to construct a possible embedding of an affine cellular algebra into its asymptotic algebra. As before, the right $R$-module structure of $J$ yields a ring homomorphism $\rho: R\rightarrow \mathrm{End}( {_RJ})$ with $a\mapsto \rho(a)=:\rho_a$ being the right multiplication of $a\in R$ on $J$, whose kernel is $\rann_R(J)$. 

Note that $c(R)\subseteq \rho^{-1}(c(\mathrm{End}(_RJ)))$ holds for the centre $c(R)$ of $R$, because for any central element $a\in c(R)$, $f\in \mathrm{End}(_RJ)$ and $b\in J$: 
$$(f\circ \rho_a - \rho_a\circ f)(b) = f(ba)-f(b)a = f(ab)-af(b)=0$$ as $f$ is left $R$-linear. This shows  $\rho_a\in c(\mathrm{End}(_RJ))$.

\begin{lemma}\label{embeddings_lemma}  Let $R$ be a ring, $J$ an ideal of $R$ and $\rho$ as above. 
	\begin{enumerate}
		\item If $\rann_R(J)=0$, then $R\hookrightarrow\mathrm{End}(_RJ)$ and $c(R)=\rho^{-1}(c(\mathrm{End}(_RJ)))$.
		\item Let $\Phi: R\rightarrow R/J \times \mathrm{End}({_RJ})$ be the ring homomorphism given by $\Phi(a)=(a+J, \rho_a)$ for all $a\in R$. Then
		$\Phi$ is an embedding if and only if $\rann_J(J)=0$. In this case  $$c(R)=\Phi^{-1}(c(R/J)\times c(\mathrm{End}({_RJ}))).$$		
	\end{enumerate}
\end{lemma}

\begin{proof}
	(1) Suppose $\rann_R(J)=0$, then $\rho$ is injective and $\rho^{-1}(c(\mathrm{End}(_RJ)))\subseteq c(R)$. Since the preceding remark showed the reverse inclusion, we obtain equality.
	
	(2) The equations $\mathrm{Ker}(\Phi) = J\cap \mathrm{Ker}(\rho) = J\cap \rann_R(J)=\rann_J(J)$ show that $\Phi$ is injective if and only if $\rann_J(J)=0$.  As remarked before, $\rho_a\in c(\mathrm{End}({_RJ}))$ for $a\in c(R)$. It is clear that $a+J \in c(R/J)$. Hence $\Phi(a)\in c(R/J)\times c(\mathrm{End}({_RJ}))$. The reverse inclusion is clear in case $\Phi$ is injective.
\end{proof}

\begin{example}
	Let $k$ be a field and $R=k[x,y]$. Set $J=Rx$. then $J\simeq \wt{R}=(R,x)$ is an affine cell ideal (with respect to the identity as involution). By Theorem \ref{endomorphismring_cellideal}, $\rho:R \simeq \mathrm{End}(_RJ)$ is an isomorphism as $x$ is a non-zero divisor in the commutative ring $R$.
\end{example}

\begin{example}\label{non_noetherian_example}
	Let $R=k[x,y]/\langle xy \rangle$ and set $J=Rx$. Then $J\simeq \wt{k[x]}=(k[x],x)$ is an affine cell ideal with $\rann_R(J)=k[y]$ and $\rann_J(J)=0$. Combining Lemma \ref{embeddings_lemma} and Theorem \ref{endomorphismring_cellideal}, we obtain an embedding $\Phi: R \hookrightarrow R/J \times \mathrm{End}(_{R}J) \simeq k[y]\times k[x]$.
\end{example}

\begin{proposition}\label{Proposition_determinant}
	Let $J$ be an affine cell ideal of $A$ with cell data $B$, $n$ and $\psi$ as above.
The map $\Phi:A\rightarrow A/J \times \mathrm{End}(_AJ)$ is injective if and only if $\mathrm{det}(\psi)$ is not a zero divisor. 
In this case
		\begin{enumerate}
			\item $\rann_A(J)=\rann_A(\ov\alpha^{-1}(\psi^+))$, where $\psi^+$ denotes the adjoint matrix of $\psi$ in $M_n(B)$.
			\item $\ov\alpha^{-1}(b\psi^+)$ is central in $A$ for any non-zero divisor $b\in B$.
			\item $\mathrm{End}({_AJ}) \simeq M_n(B)$.			
			\item  $\Phi:A \hookrightarrow A/J \times M_n(B)$ is an embedding and $c(A) = \Phi^{-1}( c(A/J) \times B).$
		\end{enumerate}
\end{proposition}
 
 \begin{proof}
By \cite{Brown}*{Theorem 9.1}, $\mathrm{det}(\psi)$ is not a zero divisor in $B$ if and only if  $\psi$ is not a zero divisor in $M_n(B)$ if and only if $\rann_J(J)=0$ by Lemma \ref{annihilators}. Hence the claim follows from Lemma \ref{embeddings_lemma}.
 	
 	(1) Let $e=\ov\alpha^{-1}(\psi^+)$. Clearly $\rann_A(J) \subseteq \rann_A(e)$. For any $a\in \rann_A(e)$ and $x\in J$:
 	$$ 0=x(ea) = x (\ov\alpha^{-1}(\psi^+) a) = \ov\alpha^{-1}((\ov\alpha(x)*\psi^+)a) =\ov\alpha^{-1}( \mathrm{det}(\psi)\ov\alpha(x)a).$$
 	Since $\mathrm{det}(\psi)$ is not a zero divisor, $\ov\alpha(x)a=0$, hence $xa=0$, i.e. $\rann_A(e)=\rann_A(J)$. 
 	
 	(2) By Corollary \ref{centralelementsB}, $b\psi^+$ is central in $\wt{M_n(B)}$ for any $b\in B$.  Thus $z=\ov\alpha^{-1}(b\psi^+)$ is central in $J$. The element $z^2$ is central in $A$ because $az^2 = (az)z = z(az)= (za)z=z(za)=z^2a$ for all $a\in A$. Hence $$0=\ov\alpha(az^2-z^2a) = a(b\psi^+*b\psi^+)-(b\psi^+*b\psi^+)a  = (ab\psi^+-b\psi^+a)b\mathrm{det}(\psi),$$
 	where we use  $\psi^+*\psi^+ = \mathrm{det}(\psi)\psi^+$ and the $A$-$B$ resp. $B$-$A$-bimodule structure of $M_n(B)$. By hypothesis, $\mathrm{det}(\psi)$ and $b$ are non-zero divisors in $B$. Hence $a(b\psi^+)=(b\psi^+)a$, which means  $\ov\alpha^{-1}(b\psi^+)$  is central in $A$. 
 	
 	(3) By (2), $\psi^+$ is a central non-zero divisor in $\wt{M_n(B)}$. Hence $M_n(B)\simeq \mathring{End}(_AJ)$ by Theorem \ref{endomorphismring_cellideal}.
 	
 	(4) The claim follows from (3) and Lemma \ref{embeddings_lemma}.
\end{proof}

%\begin{corollary}
%	If $B$ is a domain and $\mathrm{det}(\psi)\neq 0$, then $A\subseteq A/J\times M_n(B)$ and $\ov\alpha^{-1}(B\psi^+)$ is a subring of $c(A)$.
%\end{corollary}

\section{Affine Cellular Algebras}

\begin{definition}[Koenig-Xi, \cite{KoenigXi}*{3.13}]
An algebra $A$ (with the involution $i$) is called {\em affine cellular} if and only if there is a $k$-module
decomposition $A = J_1' \oplus J_2' \oplus \cdots \oplus J_n'$  (for some $n$) with $i(J_j') = J_j'$ for each $j$ and
such that setting $J_j = \bigoplus_{i=1}^j J_i'$ gives a chain of two-sided ideals of $A$: $0=J_0\subset J_1 \subset J_2
\subset \cdots \subset J_n = A$ and for each $1\leq j\leq n$ the quotient $J_j/J_{j-1}$ is an affine cell ideal of
$A/J_{j-1}$ (with respect to the involution induced by $i$ on the quotient). We call this chain a cell chain for the
affine cellular algebra $A$. The module $\Delta_j$ is called a {\em cell lattice} for the affine cell ideal $J_j/J_{j-1}\simeq
\wt{M_{m_j}(B_j)}$ and the  algebra  $M_{m_1}\left(B_1\right)   \times  \cdots \times
M_{m_n}\left(B_n\right)$ is called the {\em asymptotic algebra} of $A$.
\end{definition}

Before applying the results of the previous sections, the following Lemma is important for embedding $A$ into its asymptotic algebra.

\begin{lemma}\label{embeddings}
	Let $R$ be a ring with ascending chain of ideals
	$0=J_{-1}\subset J_0 \subset \cdots \subset J_n=R$ and denote by 
	$\rho^k : R/J_{k-1} \rightarrow \mathrm{End}({_{R/J_{k-1}}J_k/J_{k-1}})$  the right action of $R/J_{k-1}$ on $J_k/J_{k-1}$.
	 Let $m:=\mathrm{min}\{ k \mid  \rann_{R/J_{k-1}}(J_k/J_{k-1}) = 0\}$. Then
	$$\Phi: R \longrightarrow \prod_{k=0}^{m} \mathrm{End}({_{R/J_{k-1}}J_k/J_{k-1}}), \qquad a\longmapsto (\rho_a^0,
	\ldots, \rho_a^m),$$
	is an embedding of rings if and only if 
$\rann_{J_k/J_{k-1}}(J_k/J_{k-1}) = 0$ for all $0\leq k<m$. In this case 
	$$c(R) = \Phi^{-1} \left( \prod_{k=0}^{m} c\left( \mathrm{End}({_{R/J_{k-1}}J_k/J_{k-1}}\right) \right).$$
\end{lemma}

\begin{proof}
	The proof is by induction on $m\leq n$. If $m=0$, then, by Lemma \ref{embeddings_lemma},  $\Phi=\rho^0:R\rightarrow \mathrm{End}({_{R}J_0})$ is an embedding if and only if $\rann_R(J_0)=0$ in which case $c(R)=\Phi^{-1}(c(\mathrm{End}({_{R}J_0})))$ holds.
	
Let $m\geq 0$ and suppose that the statement has been proven for all rings $R$ with chains 	$0=J_{-1}\subset J_0 \subset \cdots \subset J_{n-1}\subset J_n=R$ such that $m$ is the least non-negative integer with $\rann_{R/J_{m-1}}(J_m/J_{m-1}) = 0$. Let $R$ be a ring where such least non-negative integer is $m+1$.  By Lemma \ref{embeddings_lemma},  $\Phi':R\rightarrow R/J_0 \times \mathrm{End}({_{R}J_0})$ defined by
	$\Phi'(a)=(a+J_0, \rho^0_a)$ is injective if and only if $\rann_{J_0}(J_0)=0$ holds.
 	Using the induction hypothesis on $R/J_0$,
	$\Phi'':R/J_0 \rightarrow  \prod_{k=1}^{m+1} c\left( \mathrm{End}({_{R/J_{k-1}}J_k/J_{k-1}})\right)$
	is injective if and only if $\rann_{J_k/J_{k-1}}(J_k/J_{k-1}) = 0$ for all $1\leq k < m+1$. 
	Therefore $\Phi=(\Phi'' \times \mathrm{id})\circ\Phi'$ is injective if and only if 
	 $\rann_{J_k/J_{k-1}}(J_k/J_{k-1}) = 0$ for all $0\leq k < m+1$. 
	 
	In case one of the equivalent conditions holds, 
	$c(R)=\Phi'^{-1}(c(R/J_0)\times c(\mathrm{End}({_{R}J_0})))$ by Lemma  \ref{embeddings_lemma} and	$c(R/J) = \Phi''^{-1} \left( \prod_{k=1}^{m+1} c\left( \mathrm{End}({_{R/J_{k-1}}J_k/J_{k-1}}\right) \right)$ by the induction hypothesis. This proves the claim on $c(R)$.
	\end{proof}

By induction on the length of the cell chain of an affine cellular algebra, we deduce the following Corollary from Theorem \ref{Theorem_cell_ideals}, 
Lemma \ref{embeddings} and Proposition \ref{Proposition_determinant}:

\begin{theorem}\label{properties_aca} 
	Let $A$ be an affine cellular algebra with cellular structure
	$$ 0=J_{-1} \subset J_0 \subset J_1 \subset \cdots \subset J_n = A,$$
	and $J_{j}/J_{j-1} = \wt{M_{n_j}(B_j)}=(M_{n_j}(B_j),\psi_j)$ for affine commutative $k$-algebras $B_j$ and matrices $\psi_j\in M_{n_j}(B_j)$.  %and cell lattices $\Delta_j$. 	
	Let $m=\mathrm{min}\{ k \mid \rann_{A/J_{k-1}}(J_k/J_{k-1})=0\}$.
	Then
	\begin{enumerate}
		\item $A$ satisfies a polynomial identity.
		\item The following statements are equivalent:
			\begin{enumerate}
				\item[(a)] $A/J_j$ is semiprime for all $j=0,1,\ldots, m-1$;
				\item[(b)] $B_j$ is reduced and $\mathrm{det}(\psi_j)$ is not a zero divisor in $B_j$ for all $j=0,1,\ldots, m-1$;
				\item[(c)] $\Phi: A \rightarrow \mathrm{End}\left(_{A/J_{m-1}}{J_m/J_{m-1}}\right) \times \cdots \times  \mathrm{End}(_{A}{J_0})$ is an embedding  and $B_j$ is reduced for all $j=0,1,\ldots, m-1$.
			\end{enumerate}
		In any of these cases 
		\begin{enumerate}
			\item[(i)] $\mathrm{End}(_{A/J_{j-1}}{J_j/J_{j-1}})\simeq M_{n_j}(B_j)$ for all $j=0,1,\ldots, m-1$;
			\item[(ii)] $c(A) = \Phi^{-1}\left( B_m \times \cdots \times B_0\right).$
		\end{enumerate} 
	\item If $\mathrm{det}(\psi_j)$ is invertible in $B_j$ for all $j$, then $A$ is isomorphic to its asymptotic algebra.
\end{enumerate}
\end{theorem}

\begin{proof}
If the length of the cell chain is $n=0$, then $A=J_0\simeq M_{n_0}(B)$ is a semiprime (Noetherian) PI-algebra.

(1) Follows from Theorem \ref{Theorem_cell_ideals}(1) using induction. 

(2) The equivalence $(a)\Leftrightarrow (b)$ follows from Theorem \ref{Theorem_cell_ideals}(2) by using induction. The equivalence $(b)\Leftrightarrow (c)$ follows from Lemma \ref{embeddings} and Lemma \ref{annihilators}. 
Statement (2.i) follows from Proposition \ref{Proposition_determinant}(3), while (2.ii) follows from Lemma \ref{embeddings}.

(3) Follows from  Proposition \ref{Proposition_finitelygenerated}(4) using induction.
\end{proof}

We deduce from Theorem \ref{properties_aca}(2.c + 2.i)  that an affine cellular algebra $A$ with $B_i$  reduced and $\mathrm{det}(\psi_j)$ being non-zero divisors in $B_j$ for all $j$, embeds into $M_{n_m}\left(B_m\right) \times \cdots \times M_{n_0}\left(B_0\right)$ which is a factor of the asymptotic algebra of $A$.

\medskip

The {\it Gelfand-Kirillov dimension} $GKdim(A)$ of an $k$-algebra $A$ over a field $k$ is a measure of the rate of growth of the algebra in terms of any generating set (for the precise definition see \cite{McConnellRobson}*{Chapter 8}). The Gelfand-Kirillov dimension $GKdim(B)$ of a commutative affine $k$-algebra $B$ coincides with the Krull dimension $Kdim(B)$ of $B$ (see \cite{McConnellRobson}*{8.2.14}). Furthermore, the GK-dimension of a matrix ring $M_n(B)$ over an affine algebra $B$ coincides with
that of $B$ (see \cite{McConnellRobson}*{8.2.7}) and the GK-dimension of a finite direct product of affine algebras is the maximum of the GK-dimensions of its factors (see \cite{McConnellRobson}*{8.3.3}). 

\begin{remark}
 Let $k$ be a field and $A$ an affine cellular $k$-algebra with cell chain of length $n$ such that $\mathrm{det}(\psi_j)$ is invertible in $B_j$ for all $j$. Then 
 $A$ is isomorphic to its asymptotic algebra, by Theorem \ref{properties_aca}, and hence  $GKdim(A) = \mathrm{max}\left(Kdim(B_1), \ldots, Kdim(B_n)\right)$.
\end{remark}

\begin{corollary}\label{GKdim} Let $k$ be a field and $A$ an affine cellular $k$-algebra with cell chain 
$$0=J_{-1} \subset J_0\subset \cdots \subset J_n=A,$$ such that $J_j/J_{j-1}\simeq (M_{m_j}(B_j),\psi_j)$ for $1\leq j \leq n$. Suppose $B_j$ is reduced and  $\mathrm{det}(\psi_j)$ is not a zero divisor in $B_j$ for all $j$.
Then $$GKdim(A) \leq  \mathrm{max}\left(Kdim(B_1), \ldots, Kdim(B_m)\right),$$ where $m$ is the least non-negative integer with  $\rann_{R/J_{m-1}}(J_m/J_{m-1})=0$.
If all  ideals $J_l/J_{l-1}$ with $l\leq m$ are finitely generated  left $A/J_{l-1}$-modules, then equality holds.
\end{corollary}

\begin{proof}
From Theorem \ref{properties_aca}, we obtain an embedding
$$ \Phi:A\hookrightarrow A' := M_{n_1}\left(B_1\right) \times \cdots \times M_{n_m}\left(B_m\right).$$ As before $GKdim(A')= \mathrm{max}\left(Kdim(B_1), \ldots, Kdim(B_m)\right).$ By \cite{McConnellRobson}*{8.2} $GKdim(A)\leq GKdim(A')$.
For any finitely generated module $M$ over $A$, one has $GKdim(\mathrm{End}(_AM))\leq GKdim(A)$ by \cite{McConnellRobson}*{8.2.9}). Since, by Proposition \ref{Proposition_determinant} $M_k(B)\simeq \mathrm{End}(_{A/J_{k-1}}J_k/J_{k-1})$, we obtain equality in case all modules $J_k/J_{k-1}$ are finitely generated $A/J_{k-1}$-modules.
\end{proof}

\begin{remark}
	 Recall from Proposition \ref{Proposition_finitelygenerated}(2) that  $J_l/J_{l-1}$ is finitely generated as left $A/J_{l-1}$-module if $B_l/\langle (\psi_l)_{ij}\rangle$ is a finitely generated $k$-module. This is the case for instance if the ideals $J_l/J_{l-1}$  are  idempotent or if $\mathrm{det}(\psi_l)$ is invertible in $B_l$ or if $B_l$ is an integral domain of Krull dimension $1$ and $\mathrm{det}(\psi_l)\neq 0$. 
	 In general, it is possible  that  affine cell ideals are not finitely generated as Proposition \ref{propositon_construction}  and Example \ref{example_non_finite_cell_ideal} shows.
\end{remark}

\section{Noetherian affine cellular Algebras}

 Example \ref{non_noetherian_example} shows that there exist affine cellular algebras with affine cell ideals that are not finitely generated as left ideals. These algebras are in particular not Noetherian and the question arises how to characterise Noetherian affine cellular algebras. 

%%% Noetherian Semiprime rings %%% 
 \begin{remark}\label{remark_Noetherian}
 Every right (resp. left) ideal $I$ of $R$ is a right (resp. left) ideal of $\wt{R}=(R,\psi)$, since $I*\wt{R} = I\psi R \subseteq I$.  Hence if $\wt{R}$ is right (resp. left) Noetherian, then so is $R$.

Let $I$ be a right ideal of $\wt{R}$. Then $\varphi(I)R=I\psi R = \{ \sum_{i} a_i \psi b_i \mid a_i\in I, b_i \in R\}$  is a right ideal of $R$. The map $I\mapsto \varphi(I)R$, which associates to a right ideal of $\wt{R}$ a right ideal of $R$, is order preserving. It is possible that this map is not injective as the following example shows. Let $K$ be a commutative ring, $R=K[x,y]$ and $\psi=x$. 
For all $n>0$, define $I_n=\sum_{i=1}^n Ky^i + xR$. These sets are ideals of $\wt{R}$, since 
$$I_n * \wt{R} =\sum_{i=1}^n xy^iR+x^2R = xyR  + x^2R \subseteq xR \subseteq I_n.$$
Thus $\varphi(I_n)R=xyR+x^2R=\varphi(I_m)R$ for all $n,m>0$. However $I_n\not\subseteq I_{n+1}$ since  $y^{n+1}\not\in I_n$. 

The ascending chain of ideals  $I_1\subset I_2 \subset \cdots $ is a proper ascending chain of ideals of $\wt{R}$ and shows that  $\wt{R}$ is not Noetherian, although $R$ is.
 \end{remark}

\begin{lemma}\label{Lemma_Noetherian} 
Let $k$ be a commutative ring, $R$ a $k$-algebra, $\psi \in R$ and $\wt{R}=(R,\psi)$.
Then $\wt{R}$ is right Noetherian if and only if $R$ is right Noetherian and $J/\varphi(J)R$ is a Noetherian $k$-module for all right ideals $J$ of $\wt{R}$.
\end{lemma}

\begin{proof}
Suppose $\wt{R}$ is right Noetherian. Then by Remark \ref{remark_Noetherian}, $R$ is also right Noetherian. Let $J$ be a right ideal of $\wt{R}$. Then any $k$-submodule $V$ of $J$ that contains $\varphi(J)R$ is also a right ideal of $\wt{R}$ since $V*R=\varphi(V)R \subseteq \varphi(J)R \subseteq V$.
Let $I_1 \subset I_2 \subset \cdots \subseteq J$ be any  ascending chain of $k$-submodules of $R$ containing $\varphi(J)R$. Then this  is also a chain of right ideals of $\wt{R}$ and must stop as $\wt{R}$ is right Noetherian. Therefore, $J/\varphi(J)R$ is a Noetherian $k$-module.

Now suppose that $R$ is right Noetherian and $J/\varphi(J)R$ is Noetherian $k$-module for all right ideals $J$ of $\wt{R}$.
Let $I_1\subseteq I_2 \subseteq \cdots $ be an ascending chain of right ideals of $\wt{R}$. Then $\varphi(I_1)R \subseteq \varphi(I_2)R\subseteq \cdots $ is an ascending chain of right ideals of $R$ and, as $R$ is right Noetherian, there exists $N>0$ such that  $\varphi(I_n)R = \varphi(I_N)R$ for all $n>N$. Let $J=\sum_{n\geq N} I_n$. Then also $\varphi(J)R = \varphi(I_N)R$ holds.
Hence the chain $I_N/\varphi(J)R \subseteq I_{N+1}/\varphi(J)R \subseteq \cdots \subseteq J/\varphi(J)R$ is an ascending chain of $k$-submodules of $J/\varphi(J)R$ which has to stop as $J/\varphi(J)R$ is a Noetherian $k$-module.

\end{proof}

\begin{example} 
	The example in Remark \ref{remark_Noetherian} shows that the Noetherian (affine cellular) algebra $R=k[x,y]$ has an affine cell ideal $J=Rx$ which is isomorphic to the non-Noetherian ring $\wt{R}=(R,x)$, even though $J$ is a Noetherian $R$-module.
\end{example}

Any affine cellular algebra is a PI-algebra which is semiprime under some suitable conditions. Any non-zero ideal of a semiprime PI-algebra contains a non-zero central element by a result of Rowen \cite{Rowen}. This applies in particular to the cell ideals $J_k$ of a semiprime affine cellular algebra. A (right) ring of fractions of a ring $R$ is an overring $Frac(R)$ of $R$ such that any non-zero divisor of $R$ is invertible in $Frac(R)$ and any element of $Frac(R)$ can be written in the form $ab^{-1}$ for some $a,b \in R$ (see \cite{McConnellRobson}*{3.1.2})). Posner's Theorem \cite{McConnellRobson}*{13.6.5} says that the ring of fractions $Frac(R)$ of a prime PI ring is obtained by inverting its non-zero central elements. One of the consequences of Posner's Theorem is the following Theorem:

\begin{theorem}[{\cite{McConnellRobson}*{13.6.14}}]\label{generalPosner}
A semiprime PI-ring $R$ is right Noetherian and finitely generated over its centre $c(R)$ if and only if $c(R)$ is a Noetherian ring.
\end{theorem}

\begin{question}
When is the centre of a semiprime affine cellular algebra Noetherian?
\end{question}

A sufficient condition for a semiprime affine cellular algebra $A$ to be Noetherian is that its centre $c(A)$ is Noetherian in which case $A$ would also be finitely generated over $c(A)$. This seems to be the case for some affine Hecke algebras. For instance, it has been argued in \cite{BrownGordonStroppel}*{5.1} using  \cite{Lusztig}*{3.11} that the centre of an (extended) affine  Hecke algebra is a polynomial ring in finitely many variables.

\begin{corollary}
	Let $A$ be an affine cellular algebra such that  $B_j$ is reduced and $\mathrm{det}(\psi_j)$ is not a zero divisor in $B_j$ for all $j$. Denote by $\Phi:A\rightarrow M_{n_1}(B_1)\times \cdots \times M_{n_m}(B_m)$ the embedding from Theorem \ref{properties_aca}. Then $A$ is Noetherian and finitely generated over its centre if and only if
	$c(A)=\Phi^{-1}(B_1\times \cdots \times B_m)$ is a Noetherian ring.
\end{corollary}

Small, Stafford and Warfield proved that any semiprime affine $k$-algebra $A$ of GK-dimension one is Noetherian and finitely generated over its centre (see \cite{SmallStaffordWarfield}).

\begin{corollary}\label{NoetherianCentre}
Let $A$ be an affine cellular algebra that is affine as $k$-algebra such that, for all $j$, $B_j$ is reduced, $Kdim(B_j)\leq 1$ and $\mathrm{det}(\psi_j)$ is not a zero divisor in $B_j$. Then $A$ is (left and right) Noetherian and finitely generated over its centre $c(A)$. Moreover,  $c(A)$ is a reduced affine $k$-algebra of Krull dimension at most one.
\end{corollary}

\begin{proof} By Theorem \ref{properties_aca}, $A$ is a semiprime $PI$-algebra.
By Corollary \ref{GKdim}, $GKdim(A)\leq 1$, as \mbox{$Kdim(B_j)\leq 1$} for all $j$. By the Small-Stafford-Warfield Theorem, $A$ is Noetherian and finitely generated over its centre $c(A)$.  Montgomery and Small have shown in \cite[Proposition 2]{MontgomerySmall} that the centre $c(A)$ of an affine $k$-algebra $A$ is itself affine over $k$ if $A$ is finitely generated over $c(A)$. Thus $c(A)$ is an affine $k$-algebra. Furthermore, $GKdim(c(A))\leq GKdim(A)\leq 1$. Since the GK-dimension coincides with the Krull dimension for commutative affine $k$-algebras,  $Kdim(c(A))\leq 1$.  Since $A$ is semiprime, $c(A)$ is reduced.
 \end{proof}

\begin{question}
When is a semiprime affine cellular algebra over $k$ affine as $k$-algebra?
\end{question}

Let $A=TL^a_n(q)$  be the affine Temperley-Lieb algebra on $n$ (even) strands with parameter $q$ over the field $k$ (see \cite[Section 2.3]{KoenigXi}). We will consider $q$ an indeterminate over $k$. The algebra $A$ is affine as $k$-algebra.  Let $J_{2j}$ be the ideal of $A$ generated by all affine diagrams with at most $2j$ through strings. Moreover, there exist a filtration $$0=J_{-2} \subset J_0 \subset J_2 \subset \cdots \subset J_{n-2} \subset J_n = A$$
with quotients $J_{2j}/J_{2(j-1)}$ isomorphic to generalised matrix rings $V_j \otimes B_j \otimes V_j \simeq \wt{M_{d_j}(B_j)}$ with an associated bilinear form given by the matrix $\psi_j$. Here $V_j$ is a finite dimensional vector space with basis consisting of affine partial diagrams and $B_j=k[x_j,x_j^{-1}]$ if $j\neq 0$ and $B_0=k[x_0]$.  

The next Lemma shows that the determinants $\mathrm{det}(\psi_j)$ are non-zero elements of $B_j[q]$.
\begin{lemma} The determinant $\mathrm{det}(\psi_j)$ is a polynomial in $q, x_j$ and $x_j^{-1}$  and, when considered as a polynomial in $q$ over $B_j$, is monic with leading term $q^{\mathrm{dim}(V_j)\cdot\frac{n-2j}{2}}$.\end{lemma}
\begin{proof}
Let $d_j=\mathrm{dim}(V_j)$.  For an affine partial diagram $v$ with $k$ horizontal edges, it is easy to see that
$\psi_j(v,v)=q^k$  so that these are precisely the values which occur on the diagonal of $\psi_j$. Moreover, for all
affine partial diagrams $w$ not equal to $v$, we get $\psi_j(v,w)=q^mb_j$ $(b_j\in B_j)$ with $m<k$. Thus $\psi_j$ is a a $d_j\times d_j$-matrix, where $q^k$ occurs on the diagonal and the exponents of powers of $q$ occurring outside the diagonal are smaller than $k$. The result now follows for example by considering the cofactor expansion of $\psi_j$ and an induction on $d_j$.
\end{proof}

Since the determinants $\mathrm{det}(\psi_j)$ are monic polynomials in $q$ over $B$, there are only finitely many specialisations of $q$ to values in $k$ such that $\mathrm{det}(\psi_j)$ is zero.

\begin{corollary} The affine Temperley-Lieb algebra $A=TL_n^a(q)$ is a semiprime Noetherian PI-algebra with $GKdim(A)=1$ for all but finitely many specialisations of the parameter $q$. Moreover, its centre $c(A)$ is an affine $k$-algebra of Krull dimension $1$, $A$ is finitely generated over $c(A)$ and embeds into its asymptotic algebra.
\end{corollary}

\begin{question}
Is the centre of a semiprime Noetherian affine cellular algebra affine?
\end{question}	

\medskip

We finish with a couple of examples and remarks:

\begin{example}[Schelter {\cite{Schelter}*{p 253}}]\label{Schelterexample}
	Let $K\leq L$ be a field extension with intermediate subfields $K_1$ and $K_2$ such that $K=K_1\cap K_2$ and
	$[L:K_i]<\infty$ for $i=1,2$. Set
	$$A=\left( \begin{array}{cc} K_1 + xL[x] & xL[x] \\ xL[x] & K_2 +xL[x]\end{array}\right) = \left( \begin{array}{cc} K_1
	& 0 \\ 0 & K_2\end{array}\right) \oplus x M_2\left(L[x]\right).$$ Let  $J_1=A$ and $J_0  = x M_2\left(L[x]\right)$. Then
	$A/J_0 \simeq K_1\times K_2 = B_1$  and $J_0 \simeq \wt{M_2\left(B_0\right)} = (M_2(B_0), \psi)$ with $B_0=L[x]$ and
	$\psi=\left(\begin{array}{cc} x & 0 \\ 0 & x \end{array}\right).$ Furthermore, let the matrix transpose be the involution. 
	The algebras $B_1$ and $B_0$ are reduced  and $\mathrm{det}(\psi) = x^2$ is not a zero divisor in $B_0$. Moreover, $A$ is a Noetherian semiprime $K$-algebra. The centre of $A$ can be easily computed as:
	$$ c(A) = \left\{ \left(\begin{array}{cc} k+xf & 0 \\ 0 & k+xf\end{array}\right) \mid k\in K, f\in L[x]\right\} \simeq
	K\oplus xL[x].$$
	The algebra $A$ is an {\it {affine} cellular algebra} with the cell ideal chain $J_0\subseteq J_1=A$ if and only if  $B_0$ and $B_1$ are affine $k$-algebras. Moreover, $B_0$ and $B_1$ are affine $k$-algebra  if and only if $B_0=L[x]$ is an affine $K$-algebra if and only if $L$ is an affine $K$-algebra if and only if  $[L:K]<\infty$.
	% by Noether normalisation. 
	%The latter condition, looking at the description of the centre of $A$, is equivalent to $A$ being finitely generated over its centre.
	%In other words, if $L$ is a transcendental extension of $K$, then $c(A)$ is not Noetherian, although $A$ is.
\end{example}

\begin{remark}\label{example1} By Goldie's Theorem \cite{McConnellRobson}*{2.3.6}, any semiprime right Noetherian ring $R$ has a semisimple Artinian ring of fractions $\Frac(R)$. As a consequence of Goldie's Theorem, we obtain an embedding of a Noetherian semiprime affine cellular algebra $A$  into a finite direct product of matrix rings over division rings that are finite dimensional over their centres.
   These division algebras might not be affine anymore since their centres might be transcendental field extensions of the base field.
   
 For a concrete example, let $k$ be a field and $A=k[x]$. Then $A$ is an affine cellular $k$-algebra by taking $J_1=A$ and $J_0  = x k[x]$.
Furthermore, $A/J_0 = k = B_1$ and $J_0 \simeq \wt{B_0}$ with $B_0=k[x]$ and $\psi=(x)$. The asymptotic algebra of $A$ is $k\times k[x]$, while $A$ is euqal to $k[x]$. The ring of fractions $\Frac(A)$
of $A$ is the fraction field of $k[x]$, i.e. the function field $k(x)$, which is not anymore an affine $k$-algebra.

The element $\psi$ in this example, as well as in Schelter's Example \ref{Schelterexample}, is a non-zero divisor in $A$. In
both cases $A$ is a prime ring and  Posner's Theorem \cite{McConnellRobson}*{13.6.5} says that the ring of fractions $Frac(A)$ is obtained by inverting its non-zero central elements. In particular let $A$ be  Schelter's algebra in \ref{Schelterexample} whose centre $c(A)$ is isomorphic to the diagonal matrix with entries in $K+xL[x]$. 
If $[L:K]$ is finite, then $K+xL[x] \subseteq L[x]$ is a finite ring extension and hence an integral extension. Thus the fraction field of $L[x]$ can be obtained by inverting all the non-zero elements of $K+xL[x]$ as the extension is integral, i.e. $\Frac(c(A)) \simeq \Frac(L[x]) = L(x)$. Since, by Posner's theorem, $\Frac(A)$ can be obtained also by inverting the non-zero elements of $c(A)$, we get  $$\Frac(A) \simeq A\otimes \Frac(C) \simeq  A\otimes L(x) \simeq M_2(L(x)).$$
\end{remark}

\begin{remark}
The centre in example \ref{Schelterexample}   was $c(A)=K\oplus xL[x]$. Setting $J_1=c(A)$ and $J_0=xL[x]$ we see that $J_0=\wt{L[x]}=(L[x],x)$ and $c(A)/J_0=K$ has again an (affine) cellular structure (provided $L$ is affine).
\end{remark}

We conclude with the following question:
\begin{question}
Is the centre of an affine cellular algebra again affine cellular?
\end{question}

\section*{Acknowledgments}
Most of this work has been done during a visit of the first and third author to the Universit\"at Stuttgart in 2015/2016. Both would like to thank the members of the Institut für Algebra und Zahlentheorie for their warm hospitality. They would like to thank in particular the second author for all his help and effort that made their stay possible. The third author acknowledges financial support from FCT (Portugal) through the grant SFRH/BSAB/113788/2015 as well as from DFG Schwerpunktprogramm Darstellungstheorie 1388. Moreover, the first and third author carried out their research in the framework of CMUP (UID/MAT/00144/2013), which is funded by FCT with national (MEC) and European structural funds (FEDER), under the partnership agreement PT2020.

\end{document}